\documentclass[12pt,twoside]{amsart}
\usepackage{amsmath}
\usepackage{amssymb}
\usepackage[francais,english]{babel}
\usepackage{graphicx,epsf,amsmath}  
\usepackage{epsf,graphicx}
\usepackage{comment}
  
\newtheorem{thm}{Theorem}[section]
\newtheorem{lemma}[thm]{Lemma}

\newtheorem{cor}[thm]{Corollary}
\newtheorem{defi}[thm]{Definition}

\newtheorem*{thm*}{Theorem}

\theoremstyle{definition}

\newcommand{\RR}{\mathbb{R}}

\newcommand{\be}{\begin{equation}}
\newcommand{\ee}{\end{equation}}
\newcommand{\bee}{\begin{equation*}}
\newcommand{\eee}{\end{equation*}}
\newcommand{\bea}{\begin{eqnarray}}
\newcommand{\eea}{\end{eqnarray}}
\newcommand{\bs}{\begin{split}}
\newcommand{\es}{\end{split}}

\title[degenerate elliptic equations with measures]{Some possibly degenerate elliptic  problems with measure data and non linearity on the boundary }
\author{Thierry Gallou\"et}
\author{Yannick Sire}\thanks{
{\it TG}:
Universit\'e Aix-Marseille 1 --
LATP --
Marseille, France 
{\tt gallouet@cmi.univ-mrs.fr}\\
{\it YS}:
Universit\'e Aix-Marseille 3, Paul C\'ezanne --
LATP --
Marseille, France 
{\tt sire@cmi.univ-mrs.fr}
}

\begin{document}
\maketitle
\begin{abstract}
The goal of this paper is to study some possibly degenerate elliptic equation in a bounded domain  with a nonlinear boundary condition involving measure data. We investigate two types of problems: the first one deals with the laplacian in a bounded domain with measure supported on the domain and on the boundary. A second one deals with the same type of data but involves a degenerate weight in the equation. In both cases, the nonlinearity under consideration lies on the boundary. For the first problem, we prove an optimal regularity result, whereas for the second one the optimality is not guaranteed but we provide however regularity estimates. 
\end{abstract}
\selectlanguage{francais}
\begin{abstract}
Le but de cet article est l'\'etude d'\'equations elliptiques pouvant d\' eg\'en\'erer,  \`a donn\'ees mesures, dans un domaine born\'e, et avec nonlin\'earit\'e au bord du domaine. On \'etudie deux types de probl\`emes: un permier est une \'equation elliptique non d\'egn\'er\'ee dans un domaine born\'e  avec des doinn\'ees  mesures, support\'ees \`a la fois \`a l'int\'erieur du domaine et sur le bord de celui-ci. On traite dans une deuxi\`eme partie un probl\`eme elliptique d\'eg\'en\'er\'e. On \'etablit des r\'esultat d'existence et de r\'egularit\'e dans les deux cas. Dans les deux probl\`emes consid\'er\'es, la nonlin\'earit\'e est au bord du domaine.    
\end{abstract}
\selectlanguage{english}

\tableofcontents
\tableofcontents

\section{Introduction}

Let $\Omega$ be a smooth bounded open subset of $\RR^N$ for $N \geq 2$. Let $\left \{ \Gamma_1, \Gamma_2 \right \}$ be a measurable partition of $\partial \Omega$ such that $|\Gamma_1|>0$. Consider $\mu_1 \in \mathcal M( \overline \Omega)$ and $\mu_2 \in \mathcal M(\partial \Omega)$, two Radon measures supported on $\Omega$ and $\Gamma_2$ respectively. 

The paper is devoted to the study of the two following problems for $\gamma >1$
\begin{equation}\label{boundary1}
\left \{ 
\begin{array}{c}
\Delta u =\mu_1\,\,\mbox{in $\Omega$},\\
u=0\,\,\mbox{on $\Gamma_1$},\\
\partial_\nu u=\mu_2-|u|^{\gamma-1}u\,\,\,\mbox{on $\Gamma_2 $.}
\end{array} \right. 
\end{equation}

\begin{equation}\label{boundary2}
\left \{ 
\begin{array}{c}
\nabla \cdot (d(x,\partial \Omega)^\alpha \nabla) u =\mu_1\,\,\mbox{in $\Omega$},\\
u=0\,\,\mbox{on $\Gamma_1 $},\\
d(x,\partial \Omega)^\alpha  \partial_\nu u=\mu_2-|u|^{\gamma-1}u\,\,\,\mbox{on $\Gamma_2 $.}
\end{array} \right. 
\end{equation}
where $\alpha \in (-1,1)$ and $d(x,\partial \Omega)$ denotes the distance from a point $x \in \Omega$ to the boundary $\partial \Omega$. Notice that when $\alpha=0$, problem \eqref{boundary2} reduces to \eqref{boundary1}. The weight $d(x,\partial \Omega)^\alpha $ degenerates at the boundary either by explosion for $\alpha <0$ or to $0$ for  $\alpha >0$. We will see that there is a difference, as far as regularity is concerned , between the cases $\alpha \neq 0$ and $\alpha=0$ and we will deal with these problems seperately.

Several works have been devoted to the study of elliptic equations with non smooth data. Particularly, the following equation 
$$-\Delta u +|u|^{\gamma-1} u=f \in L^1 \,\,\,\,\mbox{in}\,\,\Omega$$
with Dirichlet boundary conditions has been investigated by several authors starting with the works of Br\'ezis and Strauss \cite{BS}. In the case of $\Omega=\RR^N$, we refer the reader to the works  \cite{BBC} and \cite{GM}. In this latter work, the authors investigate the range $\gamma \leq 1$, and prove that a growth condition on $f$ is necessary to ensure existence and uniqueness. The case of the $p-$laplacian instead of the laplacian has been investigated in \cite{BGV}. 

When $f$ is a measure, the picture is more complicated. In \cite{BG}, the authors solved the existence and regularity problem for measure data in smooth domains for Leray-Lions operator of the type $-\nabla \cdot a(x,u,\nabla u)$. However one can prove existence and regularity if and only if $f \in L^1(\Omega)+W^{-2,\gamma}(\Omega)$. This latter condition is equivalent to $|f|(A)=0$ for every borelian susbet $A$ of $\Omega$ having zero $W^{2,\gamma'}-$capacity (here $\gamma'$ is the conjugate exponent of $\gamma$).

Under some assumptions on $\gamma$, we prove that our problems \eqref{boundary1}-\eqref{boundary2} admit a solution and we study their regularity. 

\medskip

The motivation to investigate the degenerate problem \eqref{boundary2} comes from recent investigations on non local operators. Indeed, the following result has been proved by Caffarelli and Silvestre (see \cite{cafS}): Given $s\in(0,1)$, let $\alpha=1-2s\in(-1,1)$. Using variables $(x,y)\in \RR^{n+1}_+ 
:=(0,+\infty)\times\RR^N$, the space $H^s(\RR^N)$ coincides with the trace on $\partial\RR^{N+1}_{+}$ of 
$$H^1(x^\alpha):= \left\{u\in H^1_{loc}(\RR^{N+1}_{+})\;:\;\int_{\RR^{n+1}_{+}}x^\alpha\left(u^2+\left| \nabla u \right|^2 \right) dxdy<+\infty\right\}.$$  
In other words, given any 
function $u\in H^1(x^\alpha)\cap C(\overline{\RR^{N+1}_{+}})$, $v:=\left. u \right|_{\partial\RR^{N+1}_{+}}\in H^s(\RR^N)$ and there exists a constant $C=C(n,s)>0$ such that
$$
\| v\|_{H^s(\RR^N)}\le C \| u \|_{ H^1(x^\alpha)}.  
$$
So, by a standard density argument (see \cite{CPSC}), every $u\in H^1(x^\alpha)$ has a well-defined trace $v\in H^s(\RR^N)$.
Conversely, any $v\in H^s(\RR^N)$ is the trace of a function $u\in H^1(x^\alpha)$. In addition, the function $u\in H^1(x^\alpha)$ defined by
\begin{equation} \label{argmin} 
u:=\arg\min\left\{ \int_{\RR^{N+1}_{+}}x^\alpha \left| \nabla w \right|^2\;dx \; : \; \left. w \right|_{\partial\RR^{N+1}_{+}}=v\right\} 
\end{equation} 
solves the PDE 
\begin{equation}\label{bdyFrac2} 
\left \{
\begin{aligned} 
{\rm div}\, (x^\alpha \nabla u)&=0 \qquad 
{\mbox{ in $\RR^{N+1}_+$}} 
\\
u&= v  
\qquad{\mbox{ on $\partial\RR^{N+1}_+$.}}\end{aligned}\right . \end{equation} 
By standard elliptic regularity, $u$ is smooth in $\RR^{N+1}_{+}$. It turns out that 
$x^\alpha u_{x} (x,\cdot)$ converges in $H^{-s}(\RR^N)$ to a distribution $f\in H^{-s}(\RR^N)$, as $x\to 0^+$ i.e. $u$ formally solves
\begin{equation}\label{bdyFrac3} 
\left \{
\begin{aligned} 
{\rm div}\, (x^\alpha \nabla u)&=0 \qquad 
{\mbox{ in $\RR^{N+1}_+ 
$}} 
\\
-x^\alpha u_x &= f  
\qquad{\mbox{ on $\partial\RR^{N+1}_+$.}}\end{aligned}\right . \end{equation}
Consider
the Dirichlet-to-Neumann operator 
$$
\Gamma_\alpha: 
\left\{
\begin{aligned}
H^s(\RR^N)&\to H^{-s}(\RR^N)\\
v&\mapsto \Gamma_{\alpha}(v)=  f:=
-x^\alpha u_x|_{\partial \RR^{N+1}_+}, 
\end{aligned}
\right.
$$
where $u$ is the solution of \eqref{argmin}--\eqref{bdyFrac3}. 
\noindent Given $f\in H^{-s}(\RR^N)$, a function $v\in H^s(\RR^N)$ solves the equation 
\begin{equation} \label{linear} 
\frac1{d_{N,s}}(-\Delta)^{s} v=f\qquad\mbox{in $\mathbb R^N$}
\end{equation} 
if and only if its lifting $u\in H^1(x^\alpha)$ solves $u= v{\mbox{ on $\partial\RR^{N+1}_+$}}$ and
\begin{equation}\label{lifted}  
\left \{
\begin{aligned} 
{\rm div}\, (x^\alpha \nabla u)&=0 \qquad 
{\mbox{ in $\RR^{N+1}_+$}} 
\\
-x^\alpha u_x &= f  
\qquad{\mbox{ on $\partial\RR^{N+1}_+.$}}
\end{aligned}\right. 
\end{equation} 
Here $d_{N,s}$ is a normalizing constant. Equation $\eqref{linear}$ involves the fractional laplacian , which symbol is a Fourier  multiplier  $|\xi|^{2s}$. We refer the reader to \cite{landkof} for a potential-theoretic study of the fractional laplacian and Riesz kernels. 

\medskip

A quick look at the previous development shows that the weight $x^\alpha $ represents the distance of a point $(x,y) \in \RR^{N+1}_+$ to the boundary $\partial \RR^{N+1}_+$ to the power $\alpha$. It is then natural to consider a somehow "localized" version of it, namely problem \eqref{boundary2}, as a starting point of regularity study of fractional order operators. We postpone to future work the study of the equation 
$$(-\Delta)^s+|u|^{\gamma-1}u = f\,\,\,\,\mbox{in}\,\,\RR^N$$
where $f \in L_{loc}^1(\RR^N)$. This problem is more challenging since it requires local estimates, as done for instance in \cite{BGV} in the case $s=1$. More precisely, we would need local estimates independent of the radius $R$ of the following boundary problem  

\begin{equation}\label{boundary3}
\left \{ 
\begin{array}{c}
\nabla \cdot (x^\alpha \nabla u)  =\mu_1,\,\,\mbox{in $B_R^+$},\\
u=0,\,\,\mbox{on $\partial^+  B_R^+$},\\
x^\alpha  \partial_x u=\mu_2-|u|^{\gamma-1}u,\,\,\,\mbox{on $\partial^0 B_R^+ $}
\end{array} \right. 
\end{equation}
where 
$$B_R^+=\left \{ (x,y) \in \RR^+ \times \RR^N,\,\,\,|(x,y)|<R\right \},$$
$$\partial^+ B_R^+=\left \{ (x,y) \in \RR^+ \times \RR^N,\,\,\,|(x,y)|=R\right \},$$
$$\partial^0 B_R^+=\left \{ (0,y),\,\,y \in  \RR^N,\,\,\,|(0,y)|<R\right \}. $$

 It has to be noticed that the weight $d(x,\partial \Omega)^\alpha$, since $\alpha \in (-1,1)$, is degenerate at the boundary of the smooth domain $\Omega$. However, this function falls into a specific class of functions, namely $A_2$ weights introduced by Muckenhoupt \cite{muck}. Indeed, a function $w(x)$ defined on $\RR^N$ is said to be $A_2$
if the following holds: there exists $C>0$ such that 
$$\sup_B \Big \{ \frac{1}{|B|}\int_B w(x)\,dx \Big \} \Big \{ \frac{1}{|B|}\int_B w(x)^{-1}\,dx \Big \}\leq C, $$   
 
for every ball $B \subset \RR^N$. In our context the $A_2$ condition writes 
$$\Big \{ \frac{1}{R^n}\int_0^R (R-r)^\alpha r^{n-1}\,dr \Big \} \Big \{ \frac{1}{R^n}\int_0^R (R-r)^{-\alpha} r^{n-1}\,dr \Big \}\leq C, $$ 
which is clearly satisfied since $\alpha \in (-1,1). $

The class $A_2$ enjoys several very nice properties that we will describe further in the paper. First, we introduce the following notion of weak solution, valid for both problems.   
\begin{defi}
A function $u \in W^{1,1}(\Omega)$ is a weak solution of \eqref{boundary2} (or \eqref{boundary1} taking $\alpha=0$) if for all $\varphi \in C^\infty (\RR^N)$ with $\varphi=0$ on $\Gamma_1$ 
\begin{equation}\label{weak}
\int_{\Omega} d(x,\partial \Omega)^\alpha \nabla u \cdot \nabla \varphi + \int_\Omega \varphi d\mu_1=\int_{\Gamma_2}\varphi d\mu_2-\int_{\Gamma_2} |u|^{\gamma-1} u \varphi.
\end{equation} 
\end{defi}

We now state our results. 

Our main result is an existence and regularity result for \eqref{weak}. 

\begin{thm}\label{exist1}
For every $\mu_1 \in \mathcal M(\overline \Omega)$, $\mu_2 \in \mathcal M (\partial \Omega)$ supported on $\Omega$ and $\Gamma_2$ respectively, $\gamma >1$ and $\alpha=0$, there exists a weak solution $u$ to \eqref{boundary1} such that 
$$u \in \bigcap_{ 1 <q <\frac{N}{N-1}} W^{1,q}(\Omega).$$
As a consequence, the trace $Tu$ of $u$ on $\partial \Omega$ satisfies 
$$Tu \in \bigcap_{ 1  \leq q <\frac{N}{N-1}} W^{1-\frac{1}{q},q}(\partial \Omega)$$
\end{thm}

The previous theorem is optimal in the sense that one can construct measures such that $u \notin W^{1,N/N-1}(\Omega)$. The second one concerns the degenerate case. 

\begin{thm}\label{exist2}
For every $\mu_1 \in \mathcal M(\overline \Omega)$, $\mu_2 \in \mathcal M (\partial \Omega)$ supported on $\Omega$ and $\Gamma_2$ respectively, $\gamma >1$ and $\alpha \in (-1,1)$, there exists a weak solution $u$ to \eqref{boundary2} such that 
$$u \in \bigcap_{ 1  \leq q < \frac{2N+2\delta (N-1)}{2N-1+\delta}. } W^{1,q}(\Omega,d(x,\partial \Omega)^\alpha)$$
for some $\delta >0$ depending on $\Omega$ and $\alpha$. 
As a consequence, the trace $Tu$ of $u$ on $\partial \Omega$ satisfies 
$$Tu \in \bigcap_{ 1 <q< \frac{2N+2\delta(N-1)}{2N-1+\delta}.} W^{1-\frac{1+\alpha}{q},q}(\partial \Omega)$$
\end{thm}

This second theorem is not optimal, as far as regularity is concerned, because of the degeneracy of the weight close to the boundary.  

In order to prove Theorems \ref{exist1}-\ref{exist2}, we proceed in several steps: 
\begin{itemize}
\item We first approximate the problem  with smooth data. 
\item We obtain estimates with suitable dependence on the data. 
\item We pass to the limit. 
\end{itemize}

This technique is reminiscent of the technique developped in \cite{BG}, for instance. In the  degenerate case, even if the techniques are similar, this requires new ingredients from the theory of degenerate elliptic equations with $A_2$ weights, which have been developped in \cite{FKS}. For sake of clarity, we will first deal with the problem $\alpha=0$, for which an optimal result can be obtained. In a second part, we will consider $\alpha \neq 0$. 

\section{The case $\alpha=0$ and problem \eqref{boundary1}}

\subsection{Regularization of the boundary problem }

To do so we start by approximating the problem \eqref{boundary1} and consider it for smooth sequences $ \left \{ \mu^n_1 \right \}_{n \geq 0},\left \{ \mu^n_2 \right \}_{n \geq 0} \subset L^\infty$ converging to $\mu_1$ and $\mu_2$ in the weak$^*$ topology , i.e. 
$$\int_{\Omega} \varphi d\mu^n_1 \rightarrow \int_{\Omega} \varphi d\mu_1,\,\,\,\,\forall \,\,\varphi \in C(\overline \Omega)$$
$$\int_{\Gamma_2} \varphi d\mu^n_2 \rightarrow \int_{\Gamma_2} \varphi d\mu_2,\,\,\,\,\forall \,\,\varphi \in C(\partial \Omega)$$

The regularized problem then writes weakly 

\begin{equation}\label{bdy1smooth}
 \int_{\Omega} \nabla u_n \cdot \nabla \varphi + \int_\Omega \varphi \mu^n_1=\int_{\Gamma_2}\varphi \mu^n_2-\int_{\Gamma_2} |u_n|^{\gamma-1} u_n \varphi.
\end{equation}

We have the following result, as a consequence of an adaptation of standard variational techniques as in \cite{BBC}. 

\begin{thm}
Let $\gamma >1$ and $\mu_1^n,\mu_2^n$ as before. There exists a weak solution $H^1(\Omega) \bigcap L^{\infty}(\Omega)$ of \eqref{bdy1smooth}. 
\end{thm}   

\subsection{Estimates on the regularized problem}

We now estimate the solution $u_n$ and its gradient $\nabla u_n$ with a convenient dependence on the data $\mu_1^n, \mu_2^n$. 

\begin{lemma}\label{estim1}
Let $\gamma >1$ and $u_n$ be a weak solution of \eqref{bdy1smooth} with $\mu_1^n, \mu_2^n \in L^2(\Gamma_2)$. Then for every $\theta>1$, there exists  $C>0$ depending on $\theta$ and on a bound of the  $L^1$ norm of  $\mu_1^n,\mu_2^n$ such that 

\begin{equation}
\|u_n\|_{L^\gamma(\Gamma_2)} \leq C
\end{equation}
and 
\begin{equation}
\int_{\Omega} \frac{|\nabla u_n|^2}{(1+|u_n|)^\theta} \leq C . 
\end{equation}
\end{lemma} 
 
\begin{proof}
We proof follows the same lines as the one in \cite{BGV}, taking into account that the nonlinearity is on the boundary. Consider the weak formulation of \eqref{bdy1smooth}, i.e. 

\begin{equation}\label{weakreg1}
 \int_{\Omega} \nabla u_n \cdot \nabla \varphi + \int_\Omega \varphi \mu^n_1=\int_{\Gamma_2}\varphi \mu^n_2-\int_{\Gamma_2} |u_n|^{\gamma-1} u_n \varphi.
\end{equation} 
for any $\varphi$ defined as before. We consider a suitable test function for the weak formulation \eqref{weakreg1}, by considering for $\theta >0$
\begin{equation}
\phi_\theta(r)=\left \{ 
\begin{array}{c}
\int_0^r \frac{dt}{(1+t)^\theta}\,\,\mbox{if $r \geq 0$},\\
-\phi_\theta(-r)\,\,\mbox{if $r < 0$},
\end{array} \right. 
\end{equation}

Notice that $\phi_\theta$ is bounded in $\RR$. We plug the following test function in \eqref{weakreg1}  
$$\varphi= \phi_\theta(u_n),$$
which is a suitable test function by Stampacchia theorem. This gives, since $\phi_\theta$ is bounded, 
$$\int_{\Omega}\frac{|\nabla u_n|^2}{(1+|u_n|)^\theta}+ \int_{\Gamma_2} |u_n|^{\gamma-1} u_n \phi_\theta (u_n)  \leq C (\|\mu_1^n\|_{L^1(\Gamma_2)}+ \|\mu_2^n\|_{L^1(\Omega)}). $$

This gives the desired result since $|u_n|^\gamma \geq  C |u_n|^{\gamma-1} u_n \phi_\theta (u_n) $ for some $C>0$ depending on $\theta$. 

\end{proof}

From the previous estimates, we deduce regularity estimates on the gradient of $u_n$ in some Lebesgue space. This is the object of the following lemma.  
\begin{lemma}
Let $ \mu_1^n\in L^2(\Omega)$ and $\mu_2^n\in L^2(\Gamma_2)$. Consider $u_n$ a weak solution of \eqref{weakreg1} such that 
$$\int_{\Gamma_2} |u_n|^\gamma \leq C$$
for some constant $C>0$ 
and 
\begin{equation}\label{estimInterm1}
\int_{\Omega} \frac{|\nabla u_n|^2}{(1+|u_n|)^\theta} \leq D
\end{equation}
for all $\theta >1$ and some constant $D>0$ depending on $\theta$. Then, for every $q \in [1, \frac{N}{N-1})$, there exists $\tilde C$ depending on $C,D$  such that  

$$\|u_n\|_{W^{1,q}(\Omega)} \leq \tilde C .$$

\end{lemma}

\begin{proof}

We write, by H\"older inequality for $q \in [1,2)$
$$\int_\Omega |\nabla u_n |^q  =  \int_\Omega \frac{|\nabla u_n|^q}{(1+|u_n|)^{\theta q/2}} (1+|u_n|)^{\theta q/2}\leq$$
$$\Big \{ \int_\Omega \frac{|\nabla u_n|^2}{(1+|u_n|)^{\theta}} \Big \} ^{q/2} \Big \{\int_\Omega (1+|u_n|)^{\frac{\theta q}{2-q}} \Big \}^{\frac{2-q}{2}} \leq $$
$$C_1 \Big \{\int_\Omega (1+|u_n|)^{\frac{\theta q}{2-q}} \Big \}^{\frac{2-q}{2}} . $$

We have for $q^* =\frac{qN}{N-q}$

$$\int_\Omega (1+|u_n|)^{\frac{\theta q}{2-q}} \leq \varepsilon \int_\Omega |u_n|^{q^*} + C_2(\varepsilon)$$
for every $\varepsilon >0$ provided that $\theta$ is chosen such that  $\frac{\theta q}{2-q} <q^*$. We now use the Sobolev embeddings to estimate the last term: since $u_n=0$ on $\Gamma_1 \subset \partial \Omega$ and $q \in[1,2)$, we have the Sobolev embedding 
$$\|u_n\|_{L^{q^*}(\Omega)} \leq C \|\nabla u_n \|_{L^q(\Omega)} .$$
Hence, chossing $\varepsilon$ small enough, this yields to 
$$\|u\|_{L^{q^*}(\Omega)} \leq C_3$$
for some $C_3>0$. By observing that one can take $\theta$ arbitrary close to $1$, one gets that 

$$\|\nabla u_n \|_{L^q(\Omega)} \leq \tilde C$$

for every  $q\in [1,\frac{N}{N-1})$.

\end{proof}

The previous theorem admits the following corollary by just using the trace embedding. 

\begin{cor}
Let $\mu_1^n,\mu_2^n$ as before. Consider $u_n$ a weak solution of \eqref{weakreg1} such that 
$$\int_{\Gamma_2} |u_n|^\gamma \leq C$$
for some constant $C>0$ 
and 
\begin{equation}
\int_{\Omega} \frac{|\nabla u_n|^2}{(1+|u_n|)^\theta} \leq D
\end{equation}
for $\theta >1$ and some constant $D>0$ depending on $\theta$ and on a bound of  the $L^1$ norm of  $\mu_1^n,\mu_2^n$. Then the trace on $ \partial \Omega$ of $u_n$ denoted $Tu_n$ satisfies for every  $q \in (1,\frac{N}{N-1})$, the estimate 
$$\|T u_n\|_{W^{1-\frac{1}{q},q}(\partial \Omega)} \leq  \tilde C $$
for some $\tilde C$ depending $C$ and $D$. 
\end{cor}

\subsection{Passage to the limit in the regularized problem}

We now come to the proof of Theorem \ref{exist1} by passing to the limit $n \to +\infty$ in the weak formulation \eqref{weakreg1}. From the previous sections, we have that up to subsequence

\begin{equation}\label{limSob}
u_n \rightarrow  u \,\,\,\mbox{ weakly in}\,\, W^{1,q}(\Omega), q \in [1, \frac{N}{N-1}).
\end{equation}
By continuity of the trace operator, we have 
\begin{equation}\label{limSobTrace}
Tu_n \rightarrow  Tu \,\,\,\mbox{weakly in}\,\, W^{1-\frac{1}{q},q}(\partial \Omega), q \in (1, \frac{N}{N-1}).
\end{equation}
Then $Tu_n$ converges strongly to $Tu$ in $L^q(\partial \Omega)$ for some $q$ and then we have up to extraction of a subsequence
 \begin{equation}
Tu_n \to Tu \,\,\,\mbox{a.e. on } \Gamma_2.  
\end{equation}

We have now to pass to the limit in the non linear term of the weak formulation \eqref{weakreg1}. To do so, we adopt the strategy of \cite{BGV} and have to prove equi-integrability of $|u_n|^\gamma$ on $\Gamma_2$. Let $t>0$ and $r>0$ and define 
\begin{equation}
\psi(s)=\left \{ 
\begin{array}{c}
\inf((s-t)^+,1)\,\,\mbox{if $s \geq 0$},\\
-\psi(-s)\,\,\mbox{if $s < 0$},
\end{array} \right. 
\end{equation}
Plugging $\varphi=\psi(u_n)$ in \eqref{weakreg1}, we are led to the following estimate
\begin{equation*}
\int_{E^0_{n,t+1}} |u_n|^\gamma \leq \int_{E^0_{n,t}} |\mu_2^n|  + \int_{E_{n,t}} |\mu_1^n| +C\int_{E_{n,t} } |\nabla u_n|^2. 
\end{equation*}
where $$E_{n,t}=\left \{ (x,t) \in \Omega \times \RR^+\,|\, |u_n(x)|\geq t \right \}$$ and
$$E^0_{n,t}=\left \{ (x,t) \in \partial \Omega \times \RR^+\,|\, |Tu_n(x)|\geq t \right \}. $$ 
Using the $L^q$ bound for the gradient for some $q \in [1,\frac{N}{N-1})$ gives that for some $\delta >0$
\begin{equation*}
\int_{E^0_{n,t+1} } |u_n|^\gamma \leq  \int_{E^0_{n,t}} |\mu_2^n|  + \int_{E_{n,t}} |\mu_1^n| +C \mathcal{L}(E_{n,t} )^\delta 
\end{equation*}
where $\mathcal{L}$ stands for the Lebesgue measure. Due to the $L^q$ bound of $|u_n|$, we have that 
$$\mathcal{L}(E_{n,t}) \to 0$$
as $t \to +\infty$. Hence the quantity $|u_n|^\gamma$ is equi-integrable and Vitali's theorem ensures that 
$$\int_{\Gamma_2} |u_n|^{\gamma-1} u_n \to \int_{\Gamma_2} |u|^{\gamma-1} u. $$

This proves the desired result. 

\section{The degenerate case $\alpha \neq 0$ and problem \eqref{boundary2}}

We now come to the proof of Theorem \eqref{exist2}. We will not give all the proofs since some of them are almost identical but emphasize the place where we loose regularity. 

\subsection{Regularization}

Consider again smooth sequences $ \left \{ \mu^n_1 \right \}_{n \geq 0},\left \{ \mu^n_2 \right \}_{n \geq 0} $ converging to $\mu_1$ and $\mu_2$ in the weak$^*$ topology in the sense described above. The regularized problem then writes weakly 

\begin{equation}\label{bdy2smooth}
 \int_{\Omega} d(x,\partial \Omega)^\alpha \nabla u_n \cdot \nabla \varphi + \int_\Omega \varphi \mu^n_1=\int_{\Gamma_2}\varphi \mu^n_2-\int_{\Gamma_2} |u_n|^{\gamma-1} u_n \varphi.
\end{equation}

Using the Lax-Milgram theorem in \cite{FKS} together with a standard topological degree argument, one gets 

\begin{thm}
Let $\gamma >1$ and $\mu_1^n,\mu_2^n$ as before. There exists a weak solution $H^1(\Omega,d(x,\partial \Omega)^\alpha) \bigcap L^{\infty}(\Omega)$ of \eqref{bdy2smooth}. 
\end{thm} 

\subsection{Estimates}

The proof of the following lemma is identical to the proof of Lemma \eqref{estim1}

\begin{lemma}\label{estim2}
Let $\gamma >1$ and $u_n$ be a weak solution of \eqref{bdy2smooth} with $\mu_1^n,\mu_2^n \in L^2(\Omega)$ and $L^2(\Gamma_2)$ respectively. Then for every $\theta>1$, there exists a constant $C>0$ depending on a bound of the $L^1$ norms of $\mu^n_1,\mu^n_2$ such that  

\begin{equation}
\|u_n\|_{L^\gamma(\Gamma_2)} \leq C
\end{equation}
and 
\begin{equation}
\int_{\Omega} d(x,\partial \Omega)^\alpha \frac{|\nabla u_n|^2}{(1+|u_n|)^\theta} \leq C . 
\end{equation}
\end{lemma} 

Before stating the lemma concerning the regularity on the gradient of $u_n$, we recall the Sobolev embedding proved in \cite{FKS}. 

\begin{thm} (see \cite{FKS})
Let $\Omega$ be an open set of $\RR^N$ and $w$ an $A_2$ function; Then there exist constants $C >0$ and $\delta >0$ depending on $\Omega$ and the weight $w$ such that for all $u \in C^\infty_0(\Omega)$ and all $1 \leq k \leq \frac{N}{N-1} +Ê\delta$ the following holds
\begin{equation}\label{embedSob}
\|u\|_{L^{2k}(\Omega,w)} \leq C \|\nabla u \|_{L^2(\Omega,w)}
\end{equation}
\end{thm}

Remark that here the Sobolev exponent is not $\frac{2N}{N-2}$ but $\frac{2N}{N-1} +Ê2\delta$ and this is where we loose regularity in the estimates. The constant $\delta$ depends on the weight $d(x,\partial \Omega)^\alpha$ and the domain $\Omega$.

\begin{lemma}
Let $\mu_1^n$ and $\mu_2^n$ as before. Consider $u_n$ a weak solution of \eqref{bdy2smooth} such that 
$$\int_{\Gamma_2} |u_n|^\gamma \leq C$$
for some constant $C>0$ depending on $\theta$
and 
\begin{equation}\label{estimInterm}
\int_{\Omega} d(x,\partial \Omega)^\alpha \frac{|\nabla u_n|^2}{(1+|u_n|)^\theta} \leq D
\end{equation}
for $\theta >1$ and some constant $D>0$ depending on $\theta$ and on a bound of the $L^1$ norms of $\mu^n_1,\mu^n_2$. Therefore, we have that $u_n$ is bounded (in terms of $q,C,D$) in $W^{1,q}(\Omega,d(x,\partial \Omega)^\alpha)$ for every $1 \leq q < \frac{N}{N-1}+\delta$. 
\end{lemma}

\begin{proof}

We write, by H\"older inequality for $q \in [1,2)$
$$\int_\Omega d(x,\partial \Omega)^\alpha|\nabla u_n |^q  =  $$
$$\int_\Omega d(x,\partial \Omega)^{q\alpha/2} \frac{|\nabla u_n|^q}{(1+|u_n|)^{q\theta/2 }} d(x,\partial \Omega)^{\alpha-q\alpha/2}(1+|u_n|)^{\theta q/2}\leq$$
$$\Big \{ \int_\Omega d(x,\partial \Omega)^\alpha \frac{|\nabla u_n|^2}{(1+|u_n|)^{\theta}} \Big \} ^{q/2} \Big \{\int_\Omega d(x,\partial \Omega)^\alpha (1+|u_n|)^{\frac{\theta q}{2-q}} \Big \}^{\frac{2-q}{2}} \leq $$
$$D_1 \Big \{\int_\Omega d(x,\partial \Omega)^\alpha (1+|u_n|)^{\frac{\theta q}{2-q}} \Big \}^{\frac{2-q}{2}} . $$
We now use the Sobolev embedding \eqref{embedSob} to estimate the last term. Since $u_n=0$ on $\Gamma_1 \subset \partial \Omega$, we have the Sobolev embedding 
$$\|u_n\|_{L^{2k}(\Omega,d(x,\partial \Omega)^\alpha)} \leq C \|\nabla u_n \|_{L^2(\Omega,d(x,\partial \Omega)^\alpha)} $$

for all $1\leq k \leq \frac{N}{N-1}+\delta$ for some $\delta >0$. 

On the other hand from the inequality \eqref{estimInterm}, we have that $\left \{ \phi_{\theta /2}(u_n) \right \}_n$ is bounded in $H^1(\Omega,d(x,\partial \Omega)^\alpha)$. By the previous embedding, the sequence $\left \{ \phi_{\theta /2}(u_n) \right \}_n$ is also bounded in $L^{2k}(\Omega,d(x,\partial \Omega)^\alpha)$ and since, when $r$ is big enough, the function $\phi_{\theta/2}(r)$ behaves like $r^{1-\theta/2}$, we deduce that

$$|u_n|^{1-\theta/2} \in L^{2k}(\Omega,d(x,\partial \Omega)^\alpha) $$
and then 
$$|u_n|^{2k-k\theta} \in L^1(\Omega,d(x,\partial \Omega)^\alpha).$$

Therefore, the sequence $\left \{ u_n \right \}_n$ is bounded in $L^q(\Omega,d(x,\partial \Omega)^\alpha)$ if 
$$\frac{\theta q}{2-q}=2k-k\theta , i.e.$$
$$q= \frac{2 (2-\theta) k}{\theta+k(2-\theta)} >1$$
for $1< \theta <2$. Taking $\theta$ arbitrarily close to $1$, one gets 
$$q < \frac{2N+2\delta(N-1)}{2N-1+\delta}.$$

Since $d(x,\partial \Omega)^\alpha$ is integrable, one gets a bound for the last term in the previous range of $q$. 
\end{proof}

The previous theorem admits the following corollary by just using the trace embedding in \cite{nekvinda}. 

\begin{cor}
Let $\mu^1_,\mu_2^n $ as before. Consider $u_n$ a weak solution of \eqref{bdy2smooth} such that 
$$\int_{\Gamma_2} |u_n|^\gamma \leq C$$
for some constant $C>0$ 
and 
\begin{equation}
\int_{\Omega} d(x,\partial \Omega)^\alpha \frac{|\nabla u_n|^2}{(1+|u_n|)^\theta} \leq D
\end{equation}
for $\theta >1$ and some constant $D>0$ depending on $\theta$. Then the trace on $ \partial \Omega$ of $u_n$ denoted $Tu_n$ is bounded in $W^{1-\frac{1+\alpha}{q},q}(\partial \Omega)$ for every $1 < q < \frac{2N+2\delta(N-1)}{2N-1+\delta}$. 
\end{cor}

The passage to the limit works exactly as in the previous section and this proves Theorem \ref{exist2}. 

\section{Extensions of the previous result}

Thanks to the approach developped in \cite{BG}, it would be possible to extend our results to more general problems, such as 

\begin{equation}\label{boundaryExtended}
\left \{ 
\begin{array}{c}
-div A(x,\nabla u)=\mu_1,\,\,\mbox{in $\Omega$},\\
u=0,\,\,\mbox{on $\Gamma_1 \subset \partial \Omega$},\\
A(x,\nabla u) \cdot \nu =\mu_2+h(u),\,\,\,\mbox{on $\Gamma_2 \subset \partial \Omega$}
\end{array} \right. 
\end{equation}
under structural assumptions on $A(x,\nabla u)$ (as to be monotone on the gradient term (like for instance the case of the $p-$laplacian) )and structural assumptions on $h$. 

We refer the interested reader to \cite{BG} for the structural assumptions on the data. Remark that we can assume, thanks to the present work, degeneracy of the operator $A(x, .)$ close to the boundary.

\bibliographystyle{abbrv}
\bibliography{bibliofile}

\begin{thebibliography}{10}

\bibitem{BBC}
P.~Benilan, H.~Brezis, and M.~G. Crandall.
\newblock A semilinear equation in {$L\sp{1}(R\sp{N})$}.
\newblock {\em Ann. Scuola Norm. Sup. Pisa Cl. Sci. (4)}, 2(4):523--555, 1975.

\bibitem{BG}
L.~Boccardo and T.~Gallou{\"e}t.
\newblock Nonlinear elliptic and parabolic equations involving measure data.
\newblock {\em J. Funct. Anal.}, 87(1):149--169, 1989.

\bibitem{BGV}
L.~Boccardo, T.~Gallou{\"e}t, and J.~L. V{\'a}zquez.
\newblock Nonlinear elliptic equations in {${\bf R}\sp N$} without growth
  restrictions on the data.
\newblock {\em J. Differential Equations}, 105(2):334--363, 1993.

\bibitem{BS}
H.~Br{\'e}zis and W.~A. Strauss.
\newblock Semi-linear second-order elliptic equations in {$L\sp{1}$}.
\newblock {\em J. Math. Soc. Japan}, 25:565--590, 1973.

\bibitem{cafS}
L.~Caffarelli and L.~Silvestre.
\newblock An extension problem related to the fractional {L}aplacian.
\newblock {\em Commun. in PDE}, 32(8):1245, 2007.

\bibitem{CPSC}
V.~Chiad{\`o}~Piat and F.~Serra~Cassano.
\newblock Relaxation of degenerate variational integrals.
\newblock {\em Nonlinear Anal.}, 22(4):409--424, 1994.

\bibitem{FKS}
E.~B. Fabes, C.~E. Kenig, and R.~P. Serapioni.
\newblock The local regularity of solutions of degenerate elliptic equations.
\newblock {\em Comm. Partial Differential Equations}, 7(1):77--116, 1982.

\bibitem{GM}
T.~Gallou{\"e}t and J.-M. Morel.
\newblock The equation {$-\Delta u+\vert u\vert \sp {\alpha-1}u=f$}, for
  {$0\leq \alpha\leq 1$}.
\newblock {\em Nonlinear Anal.}, 11(8):893--912, 1987.

\bibitem{landkof}
N.~S. Landkof.
\newblock {\em Foundations of modern potential theory}.
\newblock Springer-Verlag, New York, 1972.
\newblock Translated from the Russian by A. P. Doohovskoy, Die Grundlehren der
  mathematischen Wissenschaften, Band 180.

\bibitem{muck}
B.~Muckenhoupt.
\newblock Weighted norm inequalities for the {H}ardy maximal function.
\newblock {\em Trans. Amer. Math. Soc.}, 165:207--226, 1972.

\bibitem{nekvinda}
A.~Nekvinda.
\newblock Characterization of traces of the weighted {S}obolev space
  {$W^{1,p}(\Omega,d^\epsilon_M)$} on {$M$}.
\newblock {\em Czechoslovak Math. J.}, 43(118)(4):695--711, 1993.

\end{thebibliography}

\end{document}